\documentclass[a4paper,11pt]{amsart}

\usepackage{mathptmx,amssymb,amscd,latexsym,eulervm}
\usepackage{amsmath}
\usepackage{amsthm}
\usepackage{amsfonts}
\usepackage{mathdots}
\usepackage[onehalfspacing]{setspace}
\usepackage{paralist}
\usepackage{aliascnt}
\usepackage[initials,lite]{amsrefs}
\usepackage[inner=2.4cm,outer=2.4cm,bottom=3.2cm]{geometry}
\usepackage{xcolor}
\definecolor{citepink}{HTML}{AA3377}
\usepackage[
  colorlinks=true,
  citecolor=citepink,
  linkcolor=blue,
  urlcolor=blue
]{hyperref}

\allowdisplaybreaks

\BibSpec{collection.article}{%
  +{}  {\PrintAuthors}                {author}
  +{,} { \textit}                     {title}
  +{.} { }                            {part}
  +{:} { \textit}                     {subtitle}
  +{,} { \PrintContributions}         {contribution}
  +{,} { \PrintConference}            {conference}
  +{}  {\PrintBook}                   {book}
  +{,} { }                            {booktitle}
  +{,} { }                            {series}
  +{, vol.} { }                       {volume}
  +{,} { }                            {publisher}
  +{,} { \PrintDateB}                 {date}
  +{,} { pp.~}                        {pages}
  +{,} { }                            {status}
  +{,} { \PrintDOI}                   {doi}
  +{,} { available at \eprint}        {eprint}
  +{}  { \parenthesize}               {language}
  +{}  { \PrintTranslation}           {translation}
  +{;} { \PrintReprint}               {reprint}
  +{.} { }                            {note}
  +{.} {}                             {transition}
  +{}  {\SentenceSpace \PrintReviews} {review}
}
\AtBeginDocument{\def\MR#1{}}

\makeatletter
\@namedef{subjclassname@2020}{\textup{2020} Mathematics Subject Classification}
\makeatother

\newtheorem{theorem}{Theorem}
\newaliascnt{proposition}{theorem}
\newtheorem{proposition}[proposition]{Proposition}
\aliascntresetthe{proposition}

\newaliascnt{corollary}{theorem}

\aliascntresetthe{corollary}

\newaliascnt{lemma}{theorem}

\aliascntresetthe{lemma}

\theoremstyle{remark}
\newaliascnt{remark}{theorem}

\aliascntresetthe{remark}

\DeclareMathOperator{\Fitt}{Fitt}

\DeclareMathOperator{\pd}{pd}
\DeclareMathOperator{\grade}{grade}

\DeclareMathOperator{\im}{im}

\newcommand{\m}{\mathfrak m}

\title[Fitting ideals and quasi-Gorenstein rings]{A Fitting Criterion for Quasi-Gorenstein Local Rings}

\author{Sora Miyashita}
\address[Miyashita]{Department of Pure and Applied Mathematics, Graduate School of Information Science and Technology, The University of Osaka, Suita, Osaka 565-0871, Japan}
\email{u804642k@ecs.osaka-u.ac.jp}

\date{\today}

\begin{document}

\begin{abstract}
We prove that a Noetherian local ring $R$ with canonical module $\omega_R$ is quasi-Gorenstein if and only if $\Fitt_1^R(\omega_R)\cong\omega_R$. This settles the expectation of Eisenbud, Ficarra, Herzog, and Moradi on canonical ideals.
\end{abstract}

\maketitle

For a finite presentation $R^m\xrightarrow{A}R^n\longrightarrow M\longrightarrow0$ over a Noetherian ring $R$, the first Fitting ideal $\Fitt_1^R(M)=I_{n-1}(A)$ is generated by the $(n-1)$-minors of $A$, with $I_0(A)=R$. It is independent of the presentation and commutes with base change \cite[Section~1.4]{BH}.

The Hilbert--Burch theorem gives $\Fitt_1^R(I)=I$ when $\grade I=\pd_R(R/I)=2$. The converse holds for $\grade I\geq2$ \cite[Theorem~3.1]{EFHM}. Here $\grade I$ is the maximal length of an $R$-regular sequence in $I$, and $\pd_R$ denotes projective dimension.

For a $d$-dimensional local ring $(R,\m)$, let $E$ be the injective hull of $R/\m$. A canonical module is a finitely generated $R$-module $\omega_R$ such that $\operatorname{Hom}_R(\omega_R,E)\cong H_\m^d(R)$, where $H_\m^d(R)$ is top local cohomology. A canonical ideal is an ideal isomorphic to $\omega_R$. The ring $R$ is quasi-Gorenstein if $\omega_R\cong R$; for Cohen--Macaulay rings this is equivalent to being Gorenstein.

Eisenbud, Ficarra, Herzog, and Moradi expected that a Cohen--Macaulay local ring with a canonical ideal $K$ satisfying $\Fitt_1^R(K)=K$ is Gorenstein \cite[Section~2]{EFHM}. They proved this for domains whose canonical module has at most two generators \cite[Proposition~2.4]{EFHM}. We prove their expectation with isomorphism in place of equality, and give a quasi-Gorenstein criterion for arbitrary local rings.

\begin{theorem}\label{thm:main}
Let $(R,\m)$ be a Noetherian local ring admitting a canonical module $\omega_R$. Then $R$ is quasi-Gorenstein if and only if $\Fitt_1^R(\omega_R)\cong\omega_R$.
\end{theorem}

An ideal is called regular if it contains a nonzerodivisor of $R$.

\begin{proposition}\label{prop:determinant}
Let $I$ be a regular ideal of a Noetherian local ring $(R,\m)$ such that every $R$-endomorphism of $I$ is multiplication by an element of $R$. Then $\pd_R I\leq1$ if and only if $\Fitt_1^R(I)\cong I$.
\end{proposition}

\begin{proof}
Let $Q$ be the localization of $R$ at all nonzerodivisors and $n$ the minimal number of generators of $I$. If $n=1$, then $I\cong R$ and both conditions hold. Assume $n\geq2$. If $\pd_R I\leq1$, a minimal free resolution of $I$ has the form $0\longrightarrow R^{n-1}\xrightarrow{A}R^n\longrightarrow I\longrightarrow0$; its ranks follow from $IQ=Q$. By the Hilbert--Burch theorem \cite[Theorem~1.4.17]{BH}, $I=aI_{n-1}(A)$ for an $R$-regular element $a$, so $I\cong\Fitt_1^R(I)$.

Conversely, suppose $\Fitt_1^R(I)\cong I$. Replacing $I$ by this isomorphic ideal preserves regularity, its endomorphism ring, and its projective dimension, so we may assume $\Fitt_1^R(I)=I$. Choose a presentation $R^m\xrightarrow{A}R^n\xrightarrow{\boldsymbol f}I\longrightarrow0$, where $\boldsymbol f=(f_1,\ldots,f_n)$, and set $Z=\ker\boldsymbol f=\im A$. Localizing gives $\boldsymbol f_Q:Q^n\longrightarrow IQ=Q$; choose $g\in Q^n$ with $\boldsymbol f_Q(g)=1$. Then $Q^n=Z_Q\oplus Qg$, where $Z_Q=Z\otimes_R Q$.

Let $B=(b_1,\ldots,b_{n-1})$ be a choice of columns of $A$, allowing repetitions, and define $\delta_B(v)=\det(B,v)$, with $v$ as the last column. After localizing at any prime of $Q$, $Z_Q$ is free on $n-1$ generators, so $\delta_B$ vanishes on $Z_Q$. Thus $\delta_B=\lambda_B\boldsymbol f_Q$ with $\lambda_B=\delta_B(g)\in Q$.

For each standard basis vector $e_i$, $\delta_B(e_i)=\lambda_Bf_i$ is a signed $(n-1)$-minor of $A$ or zero. Hence $\lambda_BI\subseteq I$. The endomorphism hypothesis gives $r\in R$ with $(\lambda_B-r)I=0$; since $IQ=Q$, we have $\lambda_B=r\in R$.

The elements $\lambda_Bf_i$ generate $I$, so we can choose $B$ such that $u=\lambda_B$ is a unit: otherwise $I=\m I$, and Nakayama's lemma would give $I=0$, contrary to regularity. Since $\det(B,g)=u$, the $b_j$ form a basis of $Z_Q$. For any column $z$ of $A$, Cramer's rule gives $z=\sum_j(\lambda_{B_j}/u)b_j$, where $B_j$ is $B$ with $b_j$ replaced by $z$. The coefficients lie in $R$, so $Z=\sum_j Rb_j$. Since the $b_j$ are independent over $Q$, they form an $R$-basis of $Z$, and $\pd_R I\leq1$.
\end{proof}

The dimension of a finitely generated $R$-module $M$ is the Krull dimension of $R$ modulo its annihilator. Serre's condition $(S_2)$ means that $\operatorname{depth}M_{\mathfrak q}\geq\min\{2,\dim M_{\mathfrak q}\}$ for every prime $\mathfrak q$ with $M_{\mathfrak q}\neq0$.

\begin{proof}[Proof of Theorem~\ref{thm:main}]
If $\omega_R\cong R$, then $\Fitt_1^R(\omega_R)=R\cong\omega_R$. Conversely, suppose that $I=\Fitt_1^R(\omega_R)$ is isomorphic to $\omega_R$. Then $\Fitt_1^R(I)=I$.

First suppose that $R$ is equidimensional and unmixed, that is, $\dim R/\mathfrak q=\dim R$ for every associated prime $\mathfrak q$. The ideal $I$ has zero annihilator and satisfies $(S_2)$ \cite[(1.6) and (1.10)]{Aoyama}, hence is regular by prime avoidance. If $I\neq R$, choose a prime $\mathfrak p$ minimal over $I$. Since $I$ is regular and $(R/I)_{\mathfrak p}$ is nonzero Artinian, the depth lemma gives $\operatorname{depth}I_{\mathfrak p}=1$. The $(S_2)$ condition gives $\operatorname{ht}\mathfrak p=1$, so $R_{\mathfrak p}$ is one-dimensional Cohen--Macaulay.

The module $I_{\mathfrak p}$ is canonical over $R_{\mathfrak p}$ \cite[Corollary~4.3]{Aoyama} and has endomorphism ring $R_{\mathfrak p}$ \cite[Proposition~1.2]{AG}. Localizing $\Fitt_1^R(I)=I$ and applying Proposition~\ref{prop:determinant} gives $\pd_{R_{\mathfrak p}}I_{\mathfrak p}\leq1$. Since $\operatorname{depth}I_{\mathfrak p}=\operatorname{depth}R_{\mathfrak p}=1$, the Auslander--Buchsbaum formula \cite[Theorem~1.3.3]{BH} shows that $I_{\mathfrak p}$ is free of rank one. Hence $I_{\mathfrak p}=\Fitt_1^{R_{\mathfrak p}}(I_{\mathfrak p})=R_{\mathfrak p}$, contradicting $I\subseteq\mathfrak p$. Thus $I=R$.

For arbitrary $R$, let $\mathfrak j$ be its largest submodule with $\dim_R\mathfrak j<d=\dim R$. Then $\overline R=R/\mathfrak j$ is equidimensional and unmixed. By \cite[(1.6), (1.8) and (1.10)]{Aoyama}, $I$ is a canonical $\overline R$-module satisfying $(S_2)$ with annihilator $\mathfrak j$ in $R$. Every associated prime of $I$ is minimal over $\mathfrak j$, so every nonzero submodule of $I$ has dimension $d$. As $\dim_R\mathfrak j<d$, we have $I\cap\mathfrak j=0$. Since $\mathfrak jI=0$, base change preserves $I$ and gives $\Fitt_1^{\overline R}(I)=(I+\mathfrak j)/\mathfrak j\cong I$. The preceding case gives $I\cong\overline R$, so $I$ is cyclic over $R$ and its first Fitting ideal is $R$. Thus $I=R$ and $\omega_R\cong R$.
\end{proof}

\noindent\textbf{Acknowledgments.} I thank Matteo Varbaro for pointing out the expectation in \cite{EFHM}. ChatGPT (OpenAI) suggested the determinant argument in Proposition~\ref{prop:determinant}, which I checked and revised.

\end{document}